\documentclass[14pt]{amsart}
\address{\newline{\normalsize Courant Institute, NYU, 251 Mercer str., New York, NY 10012, USA}\newline{\it E-mail address}:
karzhema@cims.nyu.edu}
\usepackage{amscd,amssymb}
\usepackage{amsthm,amsmath,amssymb}
\usepackage[matrix,arrow]{xy}

\makeatletter\@addtoreset{equation}{section}\makeatother

\makeatletter\@addtoreset{subsection}{equation}\makeatother

\newtheorem{theorem}[equation]{Theorem}
\newtheorem{proposition}[equation]{Proposition}
\newtheorem{lemma}[equation]{Lemma}
\newtheorem{corollary}[equation]{Corollary}

\newtheorem{theorem-definition}[equation]{Theorem\,-\,definition}

\theoremstyle{definition}
\newtheorem{example}[equation]{Example}
\newtheorem{definition}[equation]{Definition}

\theoremstyle{remark}
\newtheorem{remark}[equation]{Remark}

\thanks{{\it MS 2010 classification}: 14J28, 14J15, 14M20}


\begin{document}

\title{On polarized K3 surfaces of genus 33}

\author{Ilya Karzhemanov}

\begin{abstract}
We prove that the moduli space of smooth primitively polarized
$\mathrm{K3}$ surfaces of genus $33$ is unirational.
\end{abstract}

\sloppy

\maketitle

\bigskip

\section{Introduction}
\label{section:introduction}

Let $S$ be a $\mathrm{K3}$ surface (i.\,e. smooth projective
simply\,-\,connected surface with trivial first Chern class). One
may regard such surfaces as two\,-\,dimensional counterparts of
elliptic curves. In fact, $\mathrm{K3}$ surfaces turn out to be
favourably endowed with geometric, arithmetic and
group\,-\,theoretic properties (cf. \cite{saint},
\cite{bog-tschinkel}, \cite{nik-1}, \cite{nik-2},
\cite{dolgachev-kondo}).

In the present paper, we study birational geometry of the moduli
space of \emph{primitively polarized} $\mathrm{K3}$ surfaces over
$\mathbb{C}$. Namely, we consider the pairs $(S, L)$, where $S$ is
a complex $\mathrm{K3}$ surface and $L$ is an ample divisor on $S$
($L$ is called \emph{polarization} of $S$) corresponding to some
primitive vector in the Picard lattice $\mathrm{Pic}(S)$. Once the
integer $g := (L^{2})/2 + 1$ (called the \emph{genus} of $(S, L)$)
is fixed, the pairs $(S, L)$, when considered up to isomorphisms
preserving $L$, are parameterized by the moduli space
$\mathcal{K}_g$.

Recall that $\mathcal{K}_g$ is a quasi\,-\,projective algebraic
variety (see \cite{bayle-borel}, \cite{viehweg}). In particular,
one may study such basic questions of birational geometry for
$\mathcal{K}_g$ as rationality, unirationality, rational
connectedness, Kodaira dimension estimate, etc.

S.\,Mukai's vector bundle method, developed in order to classify
higher\,-\,dimensional Fano manifolds of Picard number $1$ and
coindex $3$ (see \cite{mukai-1}, \cite{mukai-2}), allowed him to
prove unirationality of $\mathcal{K}_{g}$ for $g \in \{2, \ldots,
10, 12, 13, 18, 20\}$ (see \cite{mukai-3}, \cite{mukai-4},
\cite{mukai-5}, \cite{mukai-6}, \cite{mukai}). At the same time,
variety $\mathcal{K}_{g}$ turned out to be non\,-\,unirational for
$g \geqslant 41$, with $g \ne 42, 45, 46$ and $48$ (see
\cite{gritsenko-hulek-sancaran}, \cite{kondo-1}, \cite{kondo-2}).

In the present paper, applying the methods (mainly) from
\cite{mukai-1} and \cite{mukai}, we prove the following:

\begin{theorem}
\label{theorem:main} The moduli space $\mathcal{K}_{33}$ is
unirational.
\end{theorem}

Let us briefly outline the proof of Theorem~\ref{theorem:main}. It
follows from \cite[Corollary 1.5]{karz-3} (cf.
Proposition~\ref{theorem:r-polarized} below) that there exists a
$\mathrm{K3}$ surface $S$ such that $\mathrm{Pic}(S)$ is generated
by some very ample divisor $H$ and a $(-2)$\,-\,curve $C$,
satisfying $(H^2) = 70$ and $H \cdot C = 2$. Thus $(S, H)$ is a
primitively polarized $\mathrm{K3}$ surface of genus $36$. (Note
that all such $(S,H)$ form a hypersurface
$\mathcal{K}_{36}^{\,\text{R}}$ in $\mathcal{K}_{36}$.)
Furthermore, the divisors $H - kC$, $1 \leqslant k\leqslant 4$,
also provide primitive polarizations on $S$ (see
Lemma~\ref{theorem:polarizations} and
Remark~\ref{remark:direct-proof}), so that the surfaces $(S, H -
kC)$ are \emph{BN general} (see
Definition~\ref{definition:bn-general},
Lemma~\ref{theorem:bn-general} and
Remark~\ref{remark:direct-proof-1}). In particular, according to
S.\,Mukai there exists a \emph{rigid} vector bundle $E_3$ on $S$
of rank $3$, such that the first Chern class $c_1(E_3)$ is equal
to $H - 4C$ and $\dim H^{0}(S, E_{3}) = 7$ (see
Theorem\,-\,definition~\ref{theorem:existence-if-rigid-v-b}). This
$E_3$ is unique and determines a morphism
$\Phi_{\scriptscriptstyle E_{3}} : S \longrightarrow G(3, 7)$ into
the Grassmannian $G(3, 7) \subset
\mathbb{P}(\bigwedge^{3}\mathbb{C}^{7})$ (cf.
Remark~\ref{remark:when-e-gives-a-morphism}). Moreover,
$\Phi_{\scriptscriptstyle E_{3}}$ coincides with embedding $S
\hookrightarrow \mathbb{P}^{12}$ given by the linear system $|H -
4C|$, and the surface $S = \Phi_{\scriptscriptstyle E_{3}}(S)
\subset G(3, 7) \cap \mathbb{P}^{12}$ can be described by explicit
equations on $G(3, 7)$ (see Theorem~\ref{theorem:mukai}). In fact,
one may run these arguments for any (BN) general polarized
$\mathrm{K3}$ surface $(S_{22},L_{22})$ of genus $12$, which
implies that the moduli space $\mathcal{K}_{12}$ is unirational
(see Remark~\ref{remark:s-12-is-unirational}).

Further, vector bundle $E_3 \otimes \mathcal{O}_{S}(C)$ turns out
to be rigid of rank $3$ as well, satisfying $c_1(E_3) \otimes
\mathcal{O}_{S}(C) = H - C$ and $\dim H^{0}(S,E_3 \otimes
\mathcal{O}_{S}(C)) = 14$ (see Lemma~\ref{theorem:wedge-e-3}).
Then again we get a morphism $\Phi_{\scriptscriptstyle E_3 \otimes
\mathcal{O}_{S}(C)} : S \longrightarrow G(3, 14) \subset
\mathbb{P}(\bigwedge^{3}\mathbb{C}^{14})$ which is the embedding
$S \hookrightarrow \mathbb{P}^{33}$ given by $|H - C|$. Working a
bit more with the space $H^{0}(S,E_3 \otimes \mathcal{O}_{S}(C))$
(cf. Proposition~\ref{theorem:key-lemma} and
Corollary~\ref{theorem:key-cor}) we find two projective subspaces,
$\Pi$ and $\Lambda$, in $\mathbb{P}(\bigwedge^{3}\mathbb{C}^{14})$
(see {\bf 3.8} in Section~\ref{section:eq-1} for their
construction) such that the following holds (cf.
Lemma~\ref{theorem:key-lemma-2}):

\begin{theorem}
\label{theorem:main-1} The surface $S = \Phi_{\scriptscriptstyle
E_3 \otimes \mathcal{O}_{S}(C)}(S)$ coincides with Zariski closure
of the locus $\widehat{S}\setminus{\Pi}$ in
$$
\widehat{S} := G\big(3,14\big) \cap \Lambda \cap \big(\lambda =
\sigma_{1} = \sigma_{2} = \sigma_{3} = 0\big)
$$
for some global sections $\lambda \in
H^{0}(G(3,14),\bigwedge^{3}\mathcal{E}_{14}) \simeq
\bigwedge^{3}\mathbb{C}^{14}$ and $\sigma_{i} \in
H^{0}(G(3,14),\bigwedge^{2}\mathcal{E}_{14}) \simeq
\bigwedge^{2}\mathbb{C}^{14}$ of the universal bundle
$\mathcal{E}_{14}$ on $G(3, 14)$.
\end{theorem}

Now, any general polarized $\mathrm{K3}$ surface $(S_{64},L_{64})$
of genus $33$ can be embedded into $G(3, 14)$ the same way as $S$
above, and represented in $G(3, 14)$ in the form similar to that
in Theorem~\ref{theorem:main-1} (see {\bf 4.1} and
Lemma~\ref{theorem:common-open-u} for details). The latter allows
one, by using the preceding results, to construct a birational map
$\mathcal{K}_{33}\dashrightarrow\mathcal{K}_{12}$ (see
Lemma~\ref{theorem:cor-def-phi} and
Proposition~\ref{theorem:shape-of-s-64}) and to finish the proof
of Theorem~\ref{theorem:main}.

Finally, the paper concludes with Remark~\ref{remark:genus-36},
where we sketch a related approach to prove unirationality of
$\mathcal{K}_{36}$.

\bigskip

\thanks{{\bf Acknowledgments.}  I would like to thank A.\,Lopez,
Yu.\,G.\,Prokhorov, and F.\,Viviani for fruitful conversations. I
am also grateful to anonymous referee for helpful remarks. The
work was supported by World Premier International Research
Initiative (WPI), MEXT, Japan, Grant-in-Aid for Scientific
Research (26887009) from Japan Mathematical Society (Kakenhi), and
by the \emph{Project TROPGEO of the European Research Council}.}.

\bigskip

\section{Preliminaries}
\label{section:preliminaries}

\refstepcounter{equation}
\subsection{}

In the present section, we recall some notions and facts about
(primitively) polarized $\mathrm{K3}$ surfaces, which will be used
through the rest of the paper (see also \cite{griff-harr},
\cite{hartshorne-ag} and \cite{isk-prok} for other standard
notation, notions and facts employed below). We also establish
several auxiliary results. The ground field will be $\mathbb{C}$.

\begin{definition}[see {\cite[Definition 3.8]{mukai}}]
\label{definition:bn-general} A polarized $\mathrm{K3}$ surface
$(\frak{S}, L)$ of genus $g$ is called \emph{BN general} if
$h^{0}(\frak{S}, L_{1})h^{0}(\frak{S}, L_{2}) < g + 1$ for all
non\,-\,trivial line bundles $L_{1}, L_{2} \in
\mathrm{Pic}(\frak{S})$ such that $L = L_1 + L_2$.

\end{definition}

\medskip

\begin{example}
\label{example:ex-of-bn-gen} Generic point in the moduli space
$\mathcal{K}_g$ corresponds to a BN general $\mathrm{K3}$. All BN
general $\mathrm{K3}$ surfaces of genus $g$ form a Zariski open
subset in $\mathcal{K}_g$.

\end{example}

\medskip

\begin{definition}
\label{definition:generated-by-global-sections} Let $W$ be a
smooth projective variety and $E$ a vector bundle on $W$. Then $E$
is said to be \emph{generated by global sections} if the natural
homomorphism of $\mathcal{O}_W$\,-\,modules
$ev_{\scriptscriptstyle E} : H^{0}(W,E) \otimes \mathcal{O}_W
\longrightarrow E$ is surjective (one identifies $E$ with its
sheaf of sections as usual).

\end{definition}

\medskip

\begin{remark}[cf. {\cite[Section 2]{mukai-1}}]
\label{remark:when-e-gives-a-morphism} In the setting of
Definition~\ref{definition:generated-by-global-sections}, if $E$
is generated by global sections, then one gets a natural morphism
$\Phi_{\scriptscriptstyle E} : W \longrightarrow G(r,N)$. Here $r
:= \mathrm{rank}(E)$, $N := h^{0}(W,E)$, and $G(r,N)$ (or
$(G(r,\mathbb{C}^N$)) is the Grassmannian of $r$\,-\,dimensional
linear subspaces in $\mathbb{C}^N$. Morphism
$\Phi_{\scriptscriptstyle E}$ sends each $x \in W$ to the subspace
$E_{x}^{\vee}\subset H^{0}(W,E)^{\vee}$ ($E_{x}^{\vee}$ is the
dual to the fiber $E_x \subset E$). In particular, we have the
equality $E = \Phi_{\scriptscriptstyle E}^{*}(\mathcal{E})$ for
the universal vector bundle $\mathcal{E}$ on $G(r,N)$, so that
$H^{0}(W,E) = H^{0}(G(r,N),\mathcal{E})$. Furthermore, if the
natural homomorphism $\bigwedge^{r}H^{0}(W,E) \longrightarrow
H^{0}(W,\bigwedge^{r}E)$ induced by the $r$\,-\,th exterior power
of $ev_{\scriptscriptstyle E}$ is surjective, then
$\Phi_{\scriptscriptstyle E}$ coincides with embedding
$\Phi_{\scriptscriptstyle|c_1(E)|} : W \hookrightarrow \mathbb{P}
:= \mathbb{P}(H^{0}(W,c_1(E))^{\vee})$ given by the linear system
$\left|c_1(E)\right|$. More precisely, the diagram
$$
\begin{array}{rcl}
\Phi_{\scriptscriptstyle E} : W&\longrightarrow G\big(r,N\big)\cap\mathbb{P}&\subset G\big(r,N\big)\\
\cap&&\qquad\cap\\
\mathbb{P}&\hookrightarrow&
\mathbb{P}\big(\bigwedge^{r}H^{0}\big(W,E\big)^{\vee}\big)
\end{array}
$$
commutes, where $G(r,N) \subset
\mathbb{P}(\bigwedge^{r}H^{0}(W,E)^{\vee})$ is embedded via
Pl\"ucker.

\end{remark}

\medskip

\begin{theorem-definition}[see {\cite{mukai}, \cite{mukai-1}, \cite{mukai-7}, \cite{mukai-8}}]
\label{theorem:existence-if-rigid-v-b} Let $(\frak{S},L)$ be a
polarized $\mathrm{K3}$ surface of genus $g$. Assume that
$(\frak{S},L)$ is BN general. Then for every pair of integers
$(r,s)$, with $g = rs$, there exists a (\emph{Gieseker}) stable
vector bundle $E_r$ on $\frak{S}$ of rank $r$, such that the
following holds:

\begin{enumerate}

\item\label{1-it} $c_1(E_r) = L$;

\smallskip

\item\label{2-it} $H^{i}(\frak{S},E_{r}) = 0$ for all $i > 0$ and $h^{0}(\frak{S},E_{r}) = r + s$;

\smallskip

\item\label{3-it} $E_r$ is generated by global sections and the natural homomorphism $\bigwedge^{r}H^{0}(\frak{S},E_{r})
\longrightarrow H^{0}(\frak{S},\bigwedge^{r}E_{r}) =
H^{0}(\frak{S},L)$ is surjective (cf.
Remark~\ref{remark:when-e-gives-a-morphism});

\smallskip

\item\label{4-it} every stable vector bundle on $\frak{S}$ which
satisfies $(\ref{1-it})$ and $(\ref{2-it})$ is isomorphic to
$E_r$.

\end{enumerate}

\smallskip

$E_r$ is called \emph{rigid} vector bundle.

\end{theorem-definition}

\medskip

\begin{remark}
\label{remark:replace-h-0-with-chi} One may replace $(\ref{2-it})$
in Theorem\,-\,definition~\ref{theorem:existence-if-rigid-v-b} by
the condition $\chi(\frak{S}, E_{r}) = r + s$.

\end{remark}

\medskip

\refstepcounter{equation}
\subsection{}
\label{subsection:pre-a-0}

Let $X$ be the Fano threefold with canonical Gorenstein
singularities and anticanonical degree $(-K_{X})^3 = 70$ (see
\cite{isk-prok}, \cite{karz-2}, \cite{karz-1}). Recall that
divisor $-K_{X}$ is very ample and the linear system $|-K_{X}|$
gives an embedding of $X$ into $\mathbb{P}^{37}$. We also have
$$
\mathrm{Pic}(X) = \mathbb{Z}\cdot K_X \qquad \mbox{and} \qquad
\mathrm{Cl}(X) = \mathbb{Z}\cdot K_X \oplus \mathbb{Z}\cdot
\widehat{E}
$$
for the quadratic cone $\widehat{E}\subset X \subset
\mathbb{P}^{37}$ (see \cite[Corollary 3.11]{karz-3}).

The next result was proved in \cite{karz-3}:

\begin{proposition}[see {\cite[Corollary 1.5]{karz-3}}]
\label{theorem:r-polarized} Generic element $S \in |-K_{X}|$ is a
$\mathrm{K3}$ surface such that the lattice $\mathrm{Pic}(S)$ is
spanned by some very ample divisor $H \sim -K_{X}\big\vert_S$ and
a $(-2)$\,-\,curve $C := \widehat{E}\big\vert_S$. Furthermore, we
have $(H^2) = 70$, $H \cdot C = 2$ and the pairs $(S,H)$ form a
unirational hypersurface
$\mathcal{K}_{36}^{\,\mathrm{R}}\subset\mathcal{K}_{36}$.

\end{proposition}

\medskip

Let $S$ be the $\mathrm{K3}$ surface as in
Proposition~\ref{theorem:r-polarized}.

\begin{lemma}
\label{theorem:polarizations} $H - 4C$ is an ample divisor on $S$.
\end{lemma}

\begin{proof}
Let $Z\subset S$ be an irreducible curve such that $(H - 4C) \cdot
Z \leqslant 0$. Write
$$
Z = aH + bC
$$
in $\mathrm{Pic}(S)$ for some $a,b \in \mathbb{Z}$. Note that $a
> 0$, since $Z \ne C$, the linear system $|m(H + C)|$ is basepoint\,-\,free for
$m \gg 1$ (it provides a contraction of $C$) and $(H + C) \cdot Z
= 72a$. On the other hand, we have
$$
0 \geqslant (H - 4C) \cdot Z = 62a + 10b,
$$
which implies that $b < -6a$. But in this case we get
$$
(Z^{2}) = 70a^2 + 4ab - 2b^2 \leqslant -26a^2 < -2
$$
--- a contradiction.

Hence $(H - 4C) \cdot Z
> 0$ for every curve $Z \subset S$. Then $H - 4C$ is ample by the Nakai\,--\,Moishezon criterion and
because of $(H - 4C)^2 = 22$.

\end{proof}

\medskip

\begin{remark}
\label{remark:direct-proof} Using the same arguments as in the
proof of Lemma~\ref{theorem:polarizations}, one can show that $H -
kC$, $1\leq k \leq 3$, is an ample divisor on $S$ as well, which
provides a polarization of genus $36 - 2k - k^2$.\footnote{~It
follows from Proposition~\ref{theorem:r-polarized} that these $(S,
H - kC)$ form hypersurfaces $\mathcal{K}_{36 - 2k -
k^2}^{\text{R}}\subset\mathcal{K}_{36 - 2k - k^2}$ (birationally)
isomorphic to $\mathcal{K}_{36}^{\mathrm{R}}$ for all $k$.} It is
also possible to see this via geometric arguments. Namely, let
$p_{1} : \mathbb{P}^{37} \dashrightarrow \mathbb{P}^{34}$ be the
linear projection from the plane $\Pi$, passing through the conic
$C$. The blowup $f_1 : Y_1 \longrightarrow X$ of $C$ resolves
indeterminacies of $p_1$ on $X$ and gives a morphism $g_1 : Y_1
\longrightarrow X_1 := p_{1}(X)$. It can be easily checked that
$Y_1$ is a weak Fano threefold and $X_1 \subset \mathbb{P}^{34}$
is an anticanonically embedded Fano threefold of genus $33$ (cf.
the proof of Proposition 6.12 in \cite{karz-2}). Moreover, we
obtain $\mathrm{Pic}(Y_{1}) = \mathbb{Z} \cdot K_{Y_{1}} \oplus
\mathbb{Z} \cdot E_{f_{1}}$, where $E_{f_{1}} \simeq \mathbb{F}_4$
is the $f_1$\,-\,exceptional divisor, and the morphism $g_1$
contracts the surface $f_{1*}^{-1}(\widehat{E})$ to a point. In
particular, the singular locus $\mathrm{Sing}(X_{1})$ consists of
a unique point, $\mathrm{Pic}(X_{1}) = \mathbb{Z} \cdot K_{X_{1}}$
and $\mathrm{Cl}(X_{1}) = \mathbb{Z} \cdot K_{X_{1}} \oplus
\mathbb{Z} \cdot E^{\,(1)}$, where $E^{\,(1)} :=
g_{1*}(E_{f_{1}})$. One can prove that $E^{\,(1)}$ is the cone
over a rational normal curve of degree $4$ such that $E^{\,(1)} =
X_1 \cap \mathbb{P}^5$. In particular, there exists a rational
normal curve $C_1 \subset X_1 \setminus{\mathrm{Sing}(X_{1})}$ of
degree $4$, with $C_1 = X_1 \cap \Pi_1$ for some linear space
$\Pi_1 \simeq \mathbb{P}^4$. Proceeding with $X_1$, $\Pi_1$, etc.
the same way as with $X$, $\Pi$, etc. above, we get three more
anticanonically embedded Fano threefolds $X_2 \subset
\mathbb{P}^{29}$, $X_3 \subset \mathbb{P}^{22}$, $X_4 \subset
\mathbb{P}^{13}$ of genera $28$, $21$, $12$, respectively, such
that $\mathrm{Sing}(X_{k})$ consists of a unique point,
$\mathrm{Pic}(X_{k}) = \mathbb{Z} \cdot K_{X_{k}}$ and
$\mathrm{Cl}(X_{k}) = \mathbb{Z} \cdot K_{X_{k}} \oplus \mathbb{Z}
\cdot E^{\,(k)}$ for all $k$, where $E^{\,(k)}$ is the cone over a
rational normal curve of degree $2 + 2k$. By construction, $S$ is
isomorphic to a surface $S_k \in |-K_{X_{k}}|$, $1 \leqslant k
\leqslant 4$. Furthermore, identifying $S$ with $S_k$, we find
that $-K_{X_{k}}\big\vert_{S_{k}} \sim H - kC$ is an ample divisor
on $S$, which provides a polarization on $S$ of genus $36 - 2k -
k^2$.

\end{remark}

\medskip

\begin{lemma}
\label{theorem:bn-general} The polarized $\mathrm{K3}$ surface
$(S,H - 4C)$ (of genus $12$) is BN general.
\end{lemma}

\begin{proof}
Suppose that
$$
H - 4C = L_1 + L_2
$$
for some non\,-\,trivial $L_{1},L_{2} \in \mathrm{Pic}(S)$ (cf.
Definition~\ref{definition:bn-general}). One may assume that both
$h^{0}(S,L_{1}),h^{0}(S,L_{2})
> 0$. Write
$$
L_{i} = a_i H + b_i C
$$
in
 $\mathrm{Pic}(S)$ for some $a_{i},b_{i} \in \mathbb{Z}$. Note
that $a_i \geqslant 0$ (cf. the proof of
Lemma~\ref{theorem:polarizations}), hence we get $a_1 = 1$, $a_2 =
0$, say. The latter implies that $b_2 \ne 0$.

Now, if $b_2 < 0$, then $h^{0}(S,L_{2}) = 0$ and we are done.
Finally, if $b_2
> 0$, then $b_{1} \leqslant -5$ and hence
$$
h^{0}(S,L_{1})h^{0}(S,L_{2}) = h^{0}(S,H + b_{1}C) < h^{0}(S,H -
4C) = 13,
$$
since $h^{0}(S,L_{2}) = h^{0}(S,b_{2}C) = 1$.

\end{proof}

\medskip

\begin{remark}
\label{remark:direct-proof-1} Using the same arguments as in the
proof of Lemma~\ref{theorem:bn-general}, one can show that the
polarized $\mathrm{K3}$ surfaces $(S,H-kC)$, $0 \leqslant k
\leqslant 3$ (cf. Remark~\ref{remark:direct-proof}), are also BN
general.

\end{remark}

\medskip

Lemmas~\ref{theorem:polarizations}, \ref{theorem:bn-general} and
Theorem\,-\,definition~\ref{theorem:existence-if-rigid-v-b} imply
that there exists a rigid rank $3$ vector bundle $E_3$ on $S$,
such that $c_1(E_3) = H - 4C$ and $h^{0}(S,E_{3}) = 7$. Then from
Remark~\ref{remark:when-e-gives-a-morphism} we get the morphism
$\Phi_{\scriptscriptstyle E_{3}} : S \longrightarrow G(3,7)\cap
\mathbb{P}^{12} \subset \mathbb{P}(\bigwedge^{3}\mathbb{C}^{7})$
which coincides with embedding $\Phi_{\scriptscriptstyle|H - 4C|}
: S \hookrightarrow \mathbb{P}^{12}$. We also have $E_3 =
\Phi_{\scriptscriptstyle E_{3}}^{*}(\mathcal{E}_{7})$ for the
universal vector bundle $\mathcal{E}_{7}$ on $G(3,7)$.

Let us recall the explicit description of the image
$\Phi_{\scriptscriptstyle E_{3}}(S)$:

\begin{theorem}[see {\cite[Theorem 5.5]{mukai}}]
\label{theorem:mukai} The surface $S = \Phi_{E_{3}}(S)$ coincides
with the locus
$$
G\big(3, 7\big) \cap \big(\lambda = 0\big) \cap \big(\sigma_1 =
\sigma_2 = \sigma_3 = 0\big)
$$
for some global sections
$$
\lambda \in \bigwedge^{3}H^{0}\big(S,E_{3}\big) = H^{0}\big(G(3,
7\big), \bigwedge^{3}\mathcal{E}_{7}\big) \simeq
\bigwedge^{3}\mathbb{C}^{7}
$$
and
$$
\sigma_{1}, \sigma_{2}, \sigma_{3} \in
\bigwedge^{2}H^{0}\big(S,E_{3}\big) = H^{0}\big(G\big(3, 7\big),
\bigwedge^{2}\mathcal{E}_{7}\big) \simeq
\bigwedge^{2}\mathbb{C}^{7}.
$$

\end{theorem}

\medskip

\begin{remark}[see \cite{mukai-1}, \cite{mukai}, \cite{mukai-6}]
\label{remark:s-12-is-unirational} One may repeat literally the
preceding considerations in the case of any BN general polarized
$\mathrm{K3}$ surface $(S_{22},L_{22})$ of genus $12$. Namely,
$S_{22}$ can be embedded into $G(3,7) \cap \mathbb{P}^{12}$, where
it coincides with the locus $G(3, 7) \cap (\alpha = 0) \cap
(\tau_1 = \tau_2 = \tau_3 = 0)$ for some $\alpha \in
\bigwedge^{3}\mathbb{C}^7$ and $\tau_1, \tau_2, \tau_3 \in
\bigwedge^{2}\mathbb{C}^7$ (so that $\mathcal{O}_{S_{22}}(L_{22})
\simeq \mathcal{O}_{G(3,7)}(1)\big\vert_{S_{22}}$). Conversely,
any such locus, for generic $\alpha$ and $\tau_i$, defines a genus
$12$ BN general polarized $\mathrm{K3}$ (cf.
Example~\ref{example:ex-of-bn-gen}). This construction also shows
that $(S_{22},L_{22})$ is uniquely determined by the
$PGL(7,\mathbb{C})$\,-\,orbits of $\alpha$ and $\tau_i$. One then
arrives at a birational map between $\mathcal{K}_{12}$ and a
$\mathbb{P}^{13}$\,-\,bundle over the orbit space
$G(3,\bigwedge^{2}\mathbb{C}^{7})//PGL(3,\mathbb{C})$. From the
latter fact it is easy to deduce that variety $\mathcal{K}_{12}$
is unirational.

\end{remark}

\medskip

\begin{remark}
\label{remark:dual-construction} Recall that there exists an
isomorphism $\delta : G(3, 7) \to G(4, 7)$ induced by the
canonical bijection between the projective spaces
$\mathbb{P}(\bigwedge^{3}\mathbb{C}^{7})$ and
$\mathbb{P}(\bigwedge^{4}\mathbb{C}^{7})$ (see \cite[Ch. III, \S
11.13]{bourbaki}). Then the composition $\Psi := \delta \circ
\Phi_{\scriptscriptstyle E_{3}} : S \longrightarrow G(4, 7)$
coincides with embedding $\Phi_{\scriptscriptstyle|H - 4C|} : S
\hookrightarrow \mathbb{P}^{12}$. Let $\mathcal{E}_{7}^*$ be the
universal vector bundle on $G(4,7)$. Then we have $\Psi =
\Phi_{\mathcal{E}_{7}^{*}\big\vert_{\scriptscriptstyle S}}$ by
identifying $S$ with $\Psi(S)$. The latter implies that
$\mathcal{E}_{7}^{*}\big\vert_{S} =: E_4$ for the rigid rank $4$
vector bundle $E_4$ on $S$, such that $c_1(E_4) = H - 4C$ and
$h^{0}(S,E_{4}) = 7$ (cf.
Theorem\,-\,definition~\ref{theorem:existence-if-rigid-v-b}).
Indeed, the inclusion $S = \Psi(S) \subset G(4,7) \cap
\mathbb{P}^{12}$ coincides with the morphism
$\Phi_{\scriptscriptstyle E_{4}}$ (see the construction in
Remark~\ref{remark:when-e-gives-a-morphism}), which gives $E_4 =
\mathcal{E}_{7}^{*}\big\vert_{S}$. Furthermore, it follows from
Theorem~\ref{theorem:mukai} that the surface $S \subset G(4,7)
\cap \mathbb{P}^{12}$ coincides with the locus
$$
(\lambda^* = 0) \cap (\sigma_{1}^* = \sigma_{2}^* = \sigma_{3}^* =
0)
$$
for some $\lambda^* \in
H^{0}(G(4,7),\bigwedge^{4}\mathcal{E}_{7}^{*})$ and $\sigma_{i}^*
\in H^{0}(G(4,7),\bigwedge^{2}\mathcal{Q}_{7}^{*})$, where
$\mathcal{Q}_{7}^*$ is the dual of the universal quotient vector
bundle on $G(4,7)$. The same applies to any BN general polarized
$\mathrm{K3}$ surface $(S_{22},L_{22})$ (cf.
Remark~\ref{remark:s-12-is-unirational}).

\end{remark}

\medskip

\refstepcounter{equation}
\subsection{}

Let us now establish several properties of the vector bundle $E_3$
on the $\mathrm{K3}$ surface $S$ introduced in
{\ref{subsection:pre-a-0}}.

\begin{proposition}
\label{theorem:rigid-3-is-h-c-stable} $E_3$ is $(H -
C)$\,-\,stable.\footnote{~We employ the terminology from \cite[Ch.
4]{friedman}.}
\end{proposition}

\begin{proof}
Firstly, since $E_3$ is stable, it is $(H - 4C)$\,-\,semistable
(see \cite[Ch. 4]{friedman}). Moreover, $E_3$ is actually $(H -
4C)$\,-\,stable, for otherwise there exists a coherent subsheaf
$F\subseteq E_3$ such that $0 < \text{rank}(F) < 3$ and

$$
c_{1}(F) \cdot (H - 4C) = \frac{22}{3}\text{rank}(F) \not\in
\mathbb{Z},
$$
which is impossible.

Further, if $E_3$ is not $(H - C)$\,-\,stable, then the $(H -
4C)$\,-\,stability of $E_3$ implies that there is a cycle $Z = aH
+ bC$, $a, b\in\mathbb{Z}$, such that

\begin{equation}
\label{first-2-inequalities} (H - 4C) \cdot Z < 0 \leqslant (H -
C) \cdot Z
\end{equation}

and

\begin{equation}
\label{last-inequality} -\frac{\text{rank}(E_{3})^2}{4}\,B(E_{3})
\leqslant (Z^2) < 0
\end{equation}
for the \emph{Bogomolov number}

$$
B(E_{3}) := 2\,\text{rank}(E_{3})\,c_{2}(E_{3}) -
(\text{rank}(E_{3}) - 1)\,c_{1}(E_{3})^{2}
$$
(see \cite[Ch. 9]{friedman}). We have:

$$
\text{rank}(E_{3}) = 3, \qquad c_{1}(E_{3}) = H - 4C, \qquad
\chi(S,E_{3}) = 7,
$$
hence $c_{2}(E_{3}) = 10$ by the Riemann\,--\,Roch formula, and so
$B(E_{3}) = 16$. Then it follows from \eqref{last-inequality} that

\begin{equation}
\label{last-inequality-1} -18 \leqslant 35a^2 + 2ab - b^2 = 36a^2
- c^2
\end{equation}
for $c := a - b$. At the same time, \eqref{first-2-inequalities}
gives

\begin{equation}
\label{first-2-inequalities-1} \frac{36}{5}a - c < 0 \leqslant 18a
- c,
\end{equation}
and thus $a, c > 0$. From \eqref{last-inequality-1} and
\eqref{first-2-inequalities-1} we obtain
$$
-18 < 36a^2 - (\frac{36}{5})^{2}a^2,
$$
or equivalently,
$$
-1 < -\frac{22}{25}a^2,
$$
i.\,e. $a = 1$. Now \eqref{first-2-inequalities-1} implies that $c
\in \{8,9,\ldots,18\}$. But for such values of $a,c$ we get
$$
-18 \leqslant (6 - c)(6 + c) < -18
$$
(see \eqref{last-inequality-1}), which is a contradiction.

Proposition~\ref{theorem:rigid-3-is-h-c-stable} is completely
proved.

\end{proof}

\medskip

\begin{lemma}
\label{theorem:non-zero-a-b-c} $E_{3}\big\vert_C =
\mathcal{O}_{\mathbb{P}^1}(a) \oplus \mathcal{O}_{\mathbb{P}^1}(b)
\oplus \mathcal{O}_{\mathbb{P}^1}(c)$ for some $a\geqslant
b\geqslant c \geqslant 0$.
\end{lemma}

\begin{proof}
Indeed, if $a < 0$, say, then $s_1 \wedge s_2 \wedge s_3 = 0$ on
$C$ for all $s_{1}, s_{2}, s_3 \in H^{0}(S, E_{3})$, which is
impossible by construction of embedding $\Phi_{\scriptscriptstyle
E_{3}}$ (cf. Remark~\ref{remark:when-e-gives-a-morphism}).

\end{proof}

\medskip

\begin{lemma}
\label{theorem:non-zero-sections} For any two (distinct) generic
global sections $s_1, s_2 \in H^{0}(S, E_{3})$, the zero locus of
the global section $s_1 \wedge s_2 \in
H^{0}(S,\bigwedge^{2}E_{3})$ is of codimension $\geqslant 2$ on
$S$.
\end{lemma}

\begin{proof}
Lemma~\ref{theorem:non-zero-a-b-c} implies that one may assume the
zero locus of $s_1 \wedge s_2$ to be of codimension $\geqslant 1$.
Suppose $s_1 \wedge s_2 = 0$ on a curve $Z \subset S$. Note that
$s_1 \wedge s_2 \ne 0$ everywhere on $C$ by
Lemma~\ref{theorem:non-zero-a-b-c} and the results from
\cite[Section 1]{mukai-4}. We get $Z \cdot C = 0$ and $Z \sim m(H
+ C)$ for some $m \in \mathbb{N}$. On the other hand, the
condition $s_1 \wedge s_2 = 0$ on $Z$ implies that $Z \subseteq
(s_1 \wedge s_2 \wedge s' = 0)$ for any $s'\in H^{0}(S, E_{3})$,
i.\,e. we have $H - 4C - m(H + C) \geqslant 0$ --- a
contradiction.

\end{proof}

\medskip

\begin{remark}
\label{remark:dual-construction-1} Running the same arguments as
in the proof of Proposition~\ref{theorem:rigid-3-is-h-c-stable}
and Lemmas~\ref{theorem:non-zero-a-b-c},
\ref{theorem:non-zero-sections}, we arrive at similar results for
the rigid vector bundle $E_4$ on $S$ (see
Remark~\ref{remark:dual-construction}). Namely, one can show that
$E_4$ is $H$\,-\,stable and for any three generic global sections
$s_1, s_2, s_3 \in H^{0}(S, E_{4})$ the zero locus of the global
section $s_1 \wedge s_2 \wedge s_3 \in
H^{0}(S,\bigwedge^{3}E_{4})$ is of codimension $\geqslant 2$ on
$S$.

\end{remark}

\bigskip

\section{Proof of Theorem~\ref{theorem:main-1}}
\label{section:eq-1}

\refstepcounter{equation}
\subsection{}

Fix $S$ and $E_3$ as in Section~\ref{section:preliminaries}.

\begin{lemma}
\label{theorem:wedge-e-3} $E_{3} \otimes \mathcal{O}_{S}(C)$ is a
rigid vector bundle of rank $3$ and such that

\begin{itemize}

\item $c_1(E_{3} \otimes \mathcal{O}_{S}(C)) = H - C$;

\smallskip

\item $h^{0}(S,E_{3} \otimes \mathcal{O}_{S}(C)) =
14$.

\end{itemize}

\end{lemma}

\begin{proof}
The equality $c_1(E_{3}) = H - 4C$ implies that
$$
c_1(E_{3} \otimes \mathcal{O}_{S}(C)) = H - C
$$
for the rank $3$ vector bundle $E_3$. Consider the exact sequence
$$
0 \longrightarrow \mathcal{O}_{S}(E_{3}) \longrightarrow
\mathcal{O}_{S}(E_{3} \otimes \mathcal{O}_{S}(C)) \longrightarrow
\mathcal{O}_{C}(E_{3}(-2)) \longrightarrow 0.
$$
From $\chi(S,E_{3}) = 7$ and the Riemann\,--\,Roch formula on $C$
we deduce
$$
\chi(S, E_{3}\otimes \mathcal{O}_{S}(E)) = \chi(S, E_{3}) +
\chi(C, E_{3}(-2)) = 7 + 7 = 14
$$
because $\deg (E_{3}(-2)) = (H - C) \cdot C = 4$ for the rank $3$
vector bundle $E_{3}(-2) := (E_{3} \otimes
\mathcal{O}_{S}(C))\big\vert_C$ on $C \simeq \mathbb{P}^1$.

Further, since $E_3$ is $(H - C)$\,-\,stable (see
Proposition~\ref{theorem:rigid-3-is-h-c-stable}), $E_{3} \otimes
\mathcal{O}_{S}(C)$ is also $(H - C)$\,-\,stable, and hence $E_{3}
\otimes \mathcal{O}_{S}(C)$ is stable (see \cite[Ch.
4]{friedman}). Then, as $(S,H - C)$ is BN general of genus $33$
(see Remark~\ref{remark:direct-proof-1}),
Theorem\,-\,definition~\ref{theorem:existence-if-rigid-v-b} and
Remark~\ref{remark:replace-h-0-with-chi} complete the proof.

\end{proof}

\medskip

Put $\widetilde{E}_3 := E_3 \otimes \mathcal{O}_{S}(C)$.
Lemma~\ref{theorem:wedge-e-3} and
Theorem\,-\,definition~\ref{theorem:existence-if-rigid-v-b} imply
that the morphism $\Phi_{\scriptscriptstyle\widetilde{E}_3} : S
\longrightarrow G(3,14) \subset
\mathbb{P}(\bigwedge^{3}\mathbb{C}^{14})$ coincides with
embedding $\Phi_{\scriptscriptstyle|H - C|} : S \hookrightarrow
\mathbb{P}^{33}$. In what follows, we identify $S$ with its image
$$
\Phi_{\scriptscriptstyle\widetilde{E}_3}(S) \subset G(3,14) \cap
\mathbb{P}^{33} \subset
\mathbb{P}\big(\bigwedge^{3}\mathbb{C}^{14}\big)
$$
and $\widetilde{E}_3$ with $\mathcal{E}_{14}\big\vert_{S}$, where
$\mathcal{E}_{14}$ is the universal vector bundle on $G(3,14)$.

\refstepcounter{equation}
\subsection{}

Let us find the defining equations for $S \subset G(3,14)$.

Choose a basis $\{s_{1}, \ldots, s_{7}\}$ in $H^{0}(S, E_{3})$ for
some global sections $s_i$ (to be specified later). Let also $t
\in H^{0}(S, \mathcal{O}_{S}(C))$ be the unique (up to a
$\mathbb{C}^*$\,-\,multiple) global section. Then $\{ts_{1},
\ldots, ts_{7}, \xi_{1}, \ldots, \xi_{7}\}$ is a basis in
$H^{0}(S, \widetilde{E}_3)$ for some global sections $\xi_i$ of
$\widetilde{E}_3$ (cf. Lemma~\ref{theorem:wedge-e-3}).

\begin{proposition}
\label{theorem:key-lemma} For every $i$, there exists $\lambda_i
\in H^{0}(S, \bigwedge^{3}E_{3})$ such that
$$
\xi_i \wedge ts_{i1} \wedge ts_{i2} = t^3\lambda_i,
$$
with appropriate $s_{i1},s_{i2}\in H^0(S,E_3)$ (depending on $i$).
\end{proposition}

\begin{proof}
Consider an affine cover $\{U_{\alpha}\}$ of $S$ such that
$E_{3}\big\vert_{U_{\alpha}}$ is trivial and $t_{\alpha} :=
t\big\vert_{U_{\alpha}} \in \mathcal{O}(U_{\alpha})$ for all
$\alpha$. Put also $s_{j,\alpha} := s_{j}\big\vert_{U_{\alpha}}$
for all $j$. By Lemma~\ref{theorem:non-zero-sections}, we can
complete $\{s_{1,\alpha}, s_{2,\alpha}\}$ to a basis of
$E_{3}\big\vert_{U_{\alpha}\setminus{\Gamma}}$ for each $\alpha$,
after choosing $s_1,s_2$ appropriately. Here $\Gamma$ is a
codimension $\geqslant 2$ subset in $S$. Then the construction of
embeddings $\Phi_{\scriptscriptstyle E_{3}}$,
$\Phi_{\scriptscriptstyle\widetilde{E}_3}$ yields

\begin{equation}
\label{local-xi-i} \xi_{i}\big\vert_{U_{\alpha}\setminus{\Gamma}}
= f_{1, \alpha}s_{1,\alpha} + f_{2, \alpha}s_{2,\alpha} +
f_{\alpha}s'_{\alpha}
\end{equation}
for all $i,\alpha$ and some $s'_{\alpha}\in H^0(S,E_3)$. (Here all
$f_{j, \alpha},f_{\alpha} \in \mathcal{O}(U_{\alpha})$ depend on
$i$, and $s'_{\alpha}$ a priori depends on both $i$ and $\alpha$.)
In particular, we have
$$
\big(\xi_i \wedge ts_1 \wedge
ts_{2}\big)\big\vert_{U_{\alpha}\setminus{\Gamma}} =
f_{\alpha}s'_{\alpha} \wedge t_{\alpha}s_{1,\alpha} \wedge
t_{\alpha}s_{2,\alpha} =: F_{i, \alpha},
$$
so that the collection $\{(F_{i, \alpha},
U_{\alpha}\setminus{\Gamma})\}$ defines a global holomorphic
section of $\mathcal{O}_{S}(H - C)$ on $S\setminus{\Gamma}$, hence
on $S$.

\begin{lemma}
\label{theorem:key-lemma-1} One can choose $s_1 =: s_{i1}$ and
$s_2 =: s_{i2}$ in such a way that $F_{i, \alpha} =
t_{\alpha}^3\lambda_i\big\vert_{U_{\alpha}}$ for all $\alpha$ and
some $\lambda_i \in H^{0}(S, \bigwedge^{3}E_{3})$ (not depending
on $\alpha$).
\end{lemma}

\begin{proof}
We have $s'_{\alpha} \wedge ts_{1} \wedge ts_{2} \in H^{0}(S,
\mathcal{O}_{S}(H - 2C))$. This implies that
$$
f_{\alpha} = \frac{F_{i, \alpha}}{s'_{\alpha} \wedge
t_{\alpha}s_{1,\alpha} \wedge t_{\alpha}s_{2,\alpha}} =
\frac{\xi_i \wedge ts_1 \wedge ts_{2}}{s'_{\alpha} \wedge ts_{1}
\wedge ts_{2}}\big\vert_{U_{\alpha}} = t_{\alpha}g_i
$$
for some meromorphic function $g_i\in\mathbb{C}(S)$ (depending on
$\xi_i$). Let us choose $g_i$ to have no poles along $C$.

Note that any global section of $\widetilde{E}_3\big\vert_C =
E_3(-2)$ extends to an appropriate $\xi_i$ (apply \eqref{2-it} of
Theorem\,-\,definition~\ref{theorem:existence-if-rigid-v-b} to the
exact sequence from the proof of Lemma~\ref{theorem:wedge-e-3}).
Furthermore, since $h^0(C,E_3(-2)) = 7$, one can see that $b + c
\leqslant 6$ and $b\leqslant 4$ in
Lemma~\ref{theorem:non-zero-a-b-c}, with $c\geqslant 2$ as well by
construction of $\Phi_{\scriptscriptstyle\widetilde{E}_3}$.

Now $s_1,s_2$, after restricting to $C$, can be brought to either
of the forms (same for both $j$) $s_{ij} = (*,*,0)$ or $s_{ij} =
(*,0,*)$. We may assume w.\,l.\,o.\,g. that $\xi_i = (0,0,1)$ and
$s_1,s_2$ are $(*,0,*)$. This yields $\xi_i\wedge s_1\wedge s_2 =
0$ on $C\cap U_{\alpha}$ and so $g_i$ does not have poles along
$C$.

Thus we obtain $\xi_i \wedge ts_1 \wedge ts_{2} = g_its'_{\alpha}
\wedge ts_{1} \wedge ts_{2} = t^3\lambda_i$, with $\lambda_i :=
g_is'_{\alpha} \wedge s_{1} \wedge s_{2}\in
H^{0}(S,\bigwedge^{3}E_{3})$, because $t^3\lambda_i = \xi_i \wedge
ts_1 \wedge ts_{2}$ is a \emph{regular} section of
$\widetilde{E}_3$ and $\lambda_i$ does not have poles along $C$.

\end{proof}

\medskip

Lemma~\ref{theorem:key-lemma-1} proves
Proposition~\ref{theorem:key-lemma}.

\end{proof}

\medskip

\begin{corollary}
\label{theorem:key-cor} For all indices $i$, there exist
$\lambda'_i, \lambda''_i \in H^{0}(S,\bigwedge^{3}E_{3})$ (in
addition to $\lambda_i$ from Proposition~\ref{theorem:key-lemma})
such that
$$
\xi_i \wedge ts_{i2} \wedge ts_{i3} - t^3\lambda'_i = \xi_i \wedge
ts_{i3} \wedge ts_{i1} - t^3\lambda''_i = 0,
$$
with appropriate linearly independent $s_{i1},s_{i2},s_{i3}\in
H^0(S,E_3)$.
\end{corollary}

\begin{proof}
We retain the notation from the proof of
Proposition~\ref{theorem:key-lemma}. Note that since
$\Phi_{\scriptscriptstyle E_3}(C)\subset\mathbb{P}^{10}\cap
G(3,7)$, one may assume that $s_1\wedge s_2\wedge s_3 = 0$ on $C$,
after possibly replacing $S$ by another surface projectively
equivalent to it. Then regarding $s_1,s_2,s_3$ as vectors in a
$3$\,-\,dimensional linear $\mathbb{C}(C)$\,-\,space, we obtain
$s_{ij} := s_j$, $1\leqslant j\leqslant 3$, such that $\xi_i\wedge
s_{i1}\wedge s_{i2} =\ldots =\xi_i\wedge s_{i3}\wedge s_{i1}= 0$
on $C\cap U_{\alpha}$. This implies the claim exactly as in the
proof of Lemma~\ref{theorem:key-lemma-1}.

\end{proof}

\medskip

\begin{remark}
\label{remark:i-not-j-tup} Given two generic $\xi_i \ne \xi_j$ as
above, we observe that the pairs $(s_{ik},s_{ik'})$, all  $1 \le
k,k' \le 3$, are distinct from the corresponding pairs
$(s_{jk},s_{jk'})$. Indeed, the only place where dependence on $i$
for $s_{i1},s_{i2}$ appears is in the proof of
Lemma~\ref{theorem:key-lemma-1} (same for $j$ and other $k$), when
we have put $\xi_i = (0,0,1)$. Using the generality of
$\xi_i,\xi_j$, it is then possible to set $\xi_j = (1,0,0)$, say,
which allows one to take $s_{j1},s_{j2}$ in the form $(*,*,0)$ for
instance. This establishes our assertion.

\end{remark}

\medskip

\refstepcounter{equation}
\subsection{}
\label{subsection:pro-a-0}

Take $\lambda, \sigma_{1}, \sigma_{2}, \sigma_3$ as in
Theorem~\ref{theorem:mukai} and write
$$
\lambda = \sum_{1 \leqslant j_{1}, j_{2}, j_{3} \leqslant 7}
\alpha_{j_{1}, j_{2}, j_{3}}~s_{j_{1}} \wedge s_{j_{2}} \wedge
s_{j_{3}}, \qquad \sigma_{r} = \sum_{1 \leqslant i, j \leqslant 7}
\alpha_{i, j}^{(r)}~s_{i} \wedge s_{j}
$$
for some $\alpha_{j_{1}, j_{2}, j_{3}}, \alpha_{i, j}^{(r)} \in
\mathbb{C}$. We may assume that
$$
\lambda \in \bigwedge^{3}H^{0}\big(S,E_{3} \otimes
\mathcal{O}_{S}\big(C\big)\big) = H^{0}\big(G(3, 14\big),
\bigwedge^{3}\mathcal{E}_{14}\big) \simeq
\bigwedge^{3}\mathbb{C}^{14}
$$
and
$$
\sigma_{1},\sigma_{2},\sigma_{3} \in
\bigwedge^{2}H^{0}\big(S,E_{3} \otimes
\mathcal{O}_{S}\big(C\big)\big) = H^{0}\big(G\big(3, 14\big),
\bigwedge^{2}\mathcal{E}_{14}\big) \simeq
\bigwedge^{2}\mathbb{C}^{14}
$$
by identifying $\lambda$ and $\sigma_{r}$ with
$$
\sum_{1 \leqslant j_{1}, j_{2}, j_{3} \leqslant 7} \alpha_{j_{1},
j_{2}, j_{3}}~ts_{j_{1}} \wedge ts_{j_{2}} \wedge ts_{j_{3}}
\qquad\mbox{and}\qquad \sum_{1 \leqslant i, j \leqslant 7}
\alpha_{i, j}^{(r)}~ts_{i} \wedge ts_{j},
$$
respectively.

Let also $\Lambda$ be the linear subspace in
$\mathbb{P}(\bigwedge^{3}\mathbb{C}^{14})$ given by the equations
from Corollary~\ref{theorem:key-cor} for various $1 \leqslant i
\leqslant 7$. Then we get
$$
S \subseteq G\big(3,14\big) \cap \Lambda \cap \big(\lambda =
\sigma_{1} = \sigma_{2} = \sigma_{3} = 0\big) =: \widehat{S}
$$
by construction.

Fix $s_{ij}$ as in Corollary~\ref{theorem:key-cor} once and for
all. Let $U \subset G(3,14)$ be the open subset on which the
sections $ts_{i1}\wedge ts_{i2}\wedge ts_{i3}$, $1\leqslant
i\leqslant 7$, do not vanish simultaneously (i.\,e. $U =
G(3,14)\setminus\Pi$ as in Theorem~\ref{theorem:main-1}).

\begin{lemma}
\label{theorem:key-lemma-2} The surface $S$ coincides with
Zariski closure of the locus $\widehat{S} \cap U$.
\end{lemma}

\begin{proof}
Let us regard the sections $ts_{i},\xi_i$, $1 \leqslant i
\leqslant 7$, as vector\,-\,functions on $\widehat{S} \cap U =:
S^0$. Then, given $x\in S^0$, the construction of $\Lambda$
implies that the components of vectors $\xi_i(x)$ are uniquely
determined (via ``\,Kramer\,-\,type\,'' relations) by the
components of $s_{i}(x)$, $1 \leqslant i \leqslant 7$.

Now, restricting the natural (forgetful) map $\pi:
G(3,14)\dashrightarrow G(3,7)$ (mapping all $\xi_i$ to zero) onto
$S^0$, from the conditions $\lambda(x) = \sigma_{1}(x) =
\sigma_{2}(x) = \sigma_{3}(x) = 0$ we obtain that
$\pi\big\vert_{S^0}:S^0\dashrightarrow\pi(S^0) = S\subset G(3,7)$
is a \emph{birational} map onto a $\mathrm{K3}$ surface of genus
$12$. This implies that $S^0 \subset S$ is a surface and the
result follows.

\end{proof}

\medskip

Lemma~\ref{theorem:key-lemma-2} proves
Theorem~\ref{theorem:main-1}.

\bigskip

\section{Proof of Theorem~\ref{theorem:main}}
\label{section:eq-2}

\refstepcounter{equation}
\subsection{}

We use the same notation as in
Sections~\ref{section:preliminaries} and \ref{section:eq-1}. Let
us also fix $ts_1,\ldots,ts_7$ and $\xi_1,\ldots,\xi_7$ in what
follows as corresponding to a particular embedding
$G(3,14)\subset\mathbb{P}\big(\bigwedge^{3}\mathbb{C}^{14}\big)$.

We begin with the next supplementary

\begin{lemma}
\label{theorem:irred-comps} The irreducible (and reduced)
components of the scheme $G_{\Lambda} := G(3,14)\cap\Lambda$ are
as follows:

\begin{itemize}

    \item Zariski closure of $G_{\Lambda}\setminus\Pi$;

    \smallskip

    \item the schemes $G_m$ of dimension $12$, given (outside the rest of irreducible components of $G_{\Lambda}$) by the vanishing of various
$m$\,-\,tuples, $1\leqslant m\leqslant 7$, of sections
$ts_{i1}\wedge ts_{i2}\wedge ts_{i3}$;

    \smallskip

    \item other irreducible components contained in $I$, the indeterminacy locus of
$\pi: G(3,14)\dashrightarrow G(3,7)$ (cf. the proof of
Lemma~\ref{theorem:key-lemma-2}), given by the vanishing of all
sections $ts_i\wedge ts_j\wedge ts_k$.

\end{itemize}

\end{lemma}

\begin{proof}
The first item is clear because $G_{\Lambda}\setminus\Pi$ is
birational to $G(3,7)$. As for the remaining items, we recall that
the tuples $(ts_{ik},ts_{ik'})$, $1 \le k,k' \le 3$, are all
distinct for different $\xi_i$ (see
Remark~\ref{remark:i-not-j-tup}).

Further, on the open chart where any two of $ts_{ik} \wedge
ts_{ik'} \ne 0$, intersection $G_{\Lambda}\cap I$ has dimension
$20$, being a union of $(\mathbb{P}^1)^4$\,-\,bundles $I_{\alpha}$
over some (fixed) determinantal locus $\Gamma\subset G(3,14)$.
Now, if $ts_{11}\wedge ts_{12} = ts_{12}\wedge ts_{13} = 0$, say,
then one gets a union of
$\mathbb{P}^2\times(\mathbb{P}^1)^4$\,-\,bundles over a
codimension $6$ subset in $\Gamma$. Hence this is a \emph{proper}
closed subset in $G_{\Lambda}\cap I$ and so $G_{\Lambda}\cap I$
does not acquire any other irreducible components, besides
$I_{\alpha}$, whenever any two of $ts_{ik}\wedge ts_{ik'}$ vanish.
(The case when one (resp. two) of $ts_{ij}$ vanishes is similar
--- one gets a $\mathbb{P}^2\times(\mathbb{P}^1)^4$\,-\,bundle
(resp. $\mathbb{P}^3\times(\mathbb{P}^1)^4$\,-\,bundle) over a
codimension $3$ (resp. $6$) subset in $\Gamma$.)

Similarly, every $G_m\setminus I$ is a
$(\mathbb{P}^1)^m$\,-\,bundle over a \emph{smooth}, by generality
of $s_i$, codimension $m$ linear section $\Lambda_m$ of $G(3,7)$.
Indeed, this can be seen for $G_m\setminus I$ on the open chart
where any two of $ts_i\wedge ts_j\ne 0$, by ``\,solving the linear
equations\,'' for $\Lambda$. Now, if again some $ts_{i1}\wedge
ts_{i2} = ts_{i2}\wedge ts_{i3} = 0$, then one gets a
$\mathbb{P}^2\times (\mathbb{P}^1)^{m-1}$\,-\,bundle over a
codimension $5 + m$ subset in $G(3,7)$. Hence this is a proper
subset in $G_m$ and so $G_m$ does not acquire any irreducible
components (out of $I$) whenever any two of $ts_{ik}\wedge
ts_{ik'}$ vanish. (The case when some of $ts_{ij}$ vanish is
treated similarly.)

Finally, the loci $G_{\Lambda}\setminus\Pi$, $G_m$ and
$G_{\Lambda}\cap I_{\alpha}$ define a stratification on
$G_{\Lambda}$ by constructible irreducible subsets whose Zariski
closures remain irreducible. Thus all these constitute the
components of $G_{\Lambda}$.

\end{proof}

\medskip

\begin{remark}
\label{remark:pi-space} From (the proof of)
Lemma~\ref{theorem:irred-comps} we obtain a description of all
irreducible components of $\widehat{S}$. Namely, they are the
surface $S$, (some of) the $20$\,-\,dimensional schemes
$\widehat{S}\cap I_{\alpha} = G_{\Lambda}\cap I_{\alpha}$ (with
$\lambda$ and $\sigma_i$ being expressed via polynomials in
$ts_i\wedge ts_j\wedge ts_k$), or the loci $\widehat{S}_m :=
\widehat{S} \cap G_m$, which are $(\mathbb{P}^1)^m$\,-\,bundles
over $S\cap\Lambda_m$ (for $S$ being identified with
$\Phi_{\scriptscriptstyle E_3}(S)\subset G(3,7)$). In particular,
again by generality of $s_i$, the schemes $\widehat{S}_m$ are all
smooth and $m\leqslant 2$. This also shows that
$\widehat{S}\setminus I$ is a complete intersection (of pure
dimension $2$).

\end{remark}

\medskip

\refstepcounter{equation}
\subsection{}

Consider the rigid vector bundle $\widehat{E}_3$ on the general
$\mathrm{K3}$ surface $(S_{64},L_{64})$ of genus $33$. We have
$S_{64} = \Phi_{\scriptscriptstyle L_{64}}(S_{64}) \subset G(3,14)
\cap \mathbb{P}^{33}$ and $\widehat{E}_3 =
\mathcal{E}_{14}\big\vert_{S_{64}}$ by
Theorem\,-\,definition~\ref{theorem:existence-if-rigid-v-b}. In
particular, the pair $(S_{64}, \widehat{E}_3)$ varies in a flat
family, specializing to $(S, \widetilde{E}_3)$. Identify $S_{64}$
with the general fiber of the corresponding universal fibration
over a base $\mathcal{F}$ and denote by $S_{64}\slash\mathcal{F}$
the entire family. Similarly, considerations from
{\ref{subsection:pro-a-0}} yield a subvariety
$\mathcal{F}^{\,\text{R}} \subset \mathcal{F}$ of codimension $1$
(cf. Remark~\ref{remark:direct-proof}), so that the induced
subfamily $S\slash\mathcal{F}^{\,\text{R}} \subset
S_{64}\slash\mathcal{F}$ has $S$ as its general fiber.

The linear projection $\pi: G(3,14) \dashrightarrow G(3,7)$
induces a rational map
$$
\mathcal{F} \dashrightarrow
\big(\bigwedge^{3}H^{0}\big(S,E_{3}\big) \oplus
\big(\bigwedge^{2}H^{0}\big(S,E_{3}\big)\big)^{\oplus\,3}\big)
\supset \mathcal{F}^{\,\text{R}}
$$
which we again denote by $\pi$. Note that by construction
$\pi(\mathcal{F}^{\,\text{R}}) = \mathcal{F}^{\,\text{R}}$ at the
general point.

\begin{lemma}
\label{theorem:s-64-hat-vs-i-xxx} One has (generically)
$\pi(\mathcal{F}) = \mathcal{F}^{\,\mathrm{R}}$. In particular, a
typical $\pi$\,-\,fiber is an irreducible curve in $\mathcal{F}$,
and $\mathcal{F}^{\,\mathrm{R}}$ is a $\pi$\,-\,section.
\end{lemma}

\begin{proof}
Firstly, the restriction $\pi\big\vert_{\mathcal{F}^{\,\text{R}}}$
is smooth because, again by construction, it induces an
isomorphism on the tangent space of $\mathcal{F}^{\,\text{R}}$.

Further, if $\pi(\mathcal{F}) \ne \mathcal{F}^{\,\mathrm{R}}$ and
$\pi^{-1}(\,\mathcal{F}^{\text{R}}) \supset
\mathcal{F}^{\,\text{R}}$ strictly, then $\pi$ is (generically)
finite on $\mathcal{F}$ and there is a point in
$\pi^{-1}(\mathcal{F}^{\,\text{R}})\setminus\mathcal{F}^{\,\text{R}}$
such that the corresponding $\mathrm{K3}$ surface $S'$ admits a
rational dominant map
$$
S' \dashrightarrow S = \Phi_{\scriptscriptstyle E_3}(S)\subset
G(3,7)
$$
induced by $\pi$. Indeed, since
$\mathbb{C}(\mathcal{F})\slash\mathbb{C}(\pi(\mathcal{F}))$ is a
finite field extension generated by one primitive element, there
exists $\lambda'\in\bigwedge^{3}H^{0}\big(S,E_{3}\big)$ (resp.
$\sigma'_1, \sigma'_2,
\sigma'_3\in\bigwedge^{2}H^{0}\big(S,E_{3}\big)$) such that $S'$
coincides with Zariski closure of the locus
$$
G\big(3,14\big) \cap \Lambda \cap \big(\big(\lambda -
\lambda'\big) = \big(\sigma_{1} - \sigma'_1\big) = \big(\sigma_{2}
- \sigma'_2\big) = \big(\sigma_{3} - \sigma'_3\big) = 0\big) \cap
U
$$
(cf. Lemma~\ref{theorem:key-lemma-2}).

The restriction $\pi\big\vert_{S'}$ is given by some linear
subsystem in $|\mathcal{O}_{S'}(1)|$. Now, composing
$\pi\big\vert_{S'}$ with embedding $S =
\Phi_{\scriptscriptstyle\widetilde{E}_3}(S)\subset G(3,14)$, we
get a rational projection $S' \dashrightarrow S$ between two
$\mathrm{K3}$ in $\mathbb{P}^{33}$. Hence (non\,-\,degenerate)
$S'$ and $S$ differ by some projective automorphism. It is then
immediate that $S'$ belongs to the family
$S\slash\mathcal{F}^{\,\text{R}}$ and this is a contradiction.

Thus, under the assumption that $\pi(\mathcal{F}) \ne
\mathcal{F}^{\,\mathrm{R}}$, the map $\pi$ must be birational on
$\mathcal{F}$ (with $\pi^{-1}(\mathcal{F}^{\,\text{R}}) =
\mathcal{F}^{\,\text{R}}$). In this case, the above arguments show
that generic surface $S' := S_{64}$ projects birationally onto a
$\mathrm{K3}$ of genus $12$ (cf.
Remark~\ref{remark:s-12-is-unirational}), which is obviously
impossible.

We conclude that $\pi(\mathcal{F}) = \mathcal{F}^{\,\text{R}}$ is
the only option.

\end{proof}

One may assume w.\,l.\,o.\,g. that $S_{64}\subset\Lambda$ for
$S_{64} \subset \mathbb{P}^{33}\supset S$ and $\Lambda\subset
H^{0}\big(G\big(3, 14\big), \bigwedge^{3}\mathcal{E}_{14}\big)$
has codimension $21$. Note also that
$\bigwedge^2H^0(S,\widetilde{E}_3)$ generates a submodule
$\mathcal{O}_S^{\oplus\, N}\subseteq\bigwedge^2\widetilde{E}_3$,
some $N > 1$, which extends to a vector bundle on $\mathcal{F}$.
Hence there exist global sections
$$
\widehat{\lambda} \in H^{0}\big(G\big(3, 14\big),
\bigwedge^{3}\mathcal{E}_{14}\big)
$$
and
$$
\widehat{\sigma}_1, \widehat{\sigma}_2, \widehat{\sigma}_3 \in
H^{0}\big(G\big(3, 14\big), \bigwedge^{2}\mathcal{E}_{14}\big)
$$
such that $S_{64}$ is contained in the Zariski closure of the
locus $\widehat{S}_{64} \cap U$, where
\begin{equation}
\nonumber \widehat{S}_{64} := G_{\Lambda} \cap
\big(\widehat{\lambda} = \widehat{\sigma}_1 = \widehat{\sigma}_2 =
\widehat{\sigma}_3 = 0\big)
\end{equation}
(compare with Lemma~\ref{theorem:key-lemma-2}).

The family $S_{64}\slash\mathcal{F}$ extends to an algebraic
family $\widehat{S}_{64}\slash\widehat{\mathcal{F}}\supseteq
S_{64}\slash\mathcal{F}$ (with
$\widehat{\mathcal{F}}\supseteq\mathcal{F}$ and conventions on the
notation as earlier).

\begin{lemma}
\label{theorem:s-64-hat-vs-i} Universal fibration
$\widehat{S}_{64}\slash\widehat{\mathcal{F}} \longrightarrow
\widehat{\mathcal{F}}$ is flat at the general point.
\end{lemma}

\begin{proof}
Cutting with hyperplanes we reduce to the $2$\,-\,dimensional
case. Thus suppose there is a smooth projective surface $V$ and an
algebraic family of schemes $Z_t\subset V$. One may assume the
$0$\,-\,dimensional part of all $Z_t$ to have the same length and
local multiplicities. Then, since the Hilbert polynomial (of a
curve) has only $\mathbb{Z}$\,-\,coefficients, we obtain that the
$1$\,-\,dimensional part of $Z_t$ varies algebraically with $t$.
Summing up, all $Z_t$ have the same Hilbert polynomial, and the
claim follows.

\end{proof}

According to Lemma~\ref{theorem:s-64-hat-vs-i} we may assume the
family $\widehat{S}_{64}\slash\widehat{\mathcal{F}}$ to be flat
(for taking the \emph{flat closure} still leaves
$S_{64}\slash\mathcal{F}$ as a (flat) subfamily). One can also
restrict to $\widehat{\mathcal{F}} = \mathcal{F}$ and identify
$\mathcal{F}$ with $\mathcal{K}_{33}$ via the Luna's slice theorem
(for $S$, $S_{64}$, etc. being defined up to $\text{Aut}\,G(3,14)
= PGL(14,\mathbb{C})$).

\begin{lemma}
\label{theorem:common-open-u} There exists, same for all
$\mathrm{K3}$ surfaces in $S_{64}\slash\mathcal{F}$, Zariski open
subset $U_0\subset G(3,14)$ such that $\widehat{S}\cap U_0 = S
\cap U_0$ and $\widehat{S}_{64}\cap U_0 = S_{64} \cap U_0$.
\end{lemma}

\begin{proof}
Flatness and Remark~\ref{remark:pi-space} imply that all
irreducible components of $\widehat{S}_{64}\slash\mathcal{F}$ are
either (contained in) various $G_{\Lambda}\cap I_{\alpha}$, or
coincide with some $(\mathbb{P}^1)^m$\,-\,bundles over
$S_{64}\slash\mathcal{F}$, $m \le 2$. In particular, the dimension
of these components is at most
$$
\max\left\{20, 19 + 2 + 2\right\} = 23.
$$

Further, all irreducible components of $\widehat{S}_{64}$ (resp.
$\widehat{S}$) different from $S_{64}$ (resp. $S$) vary in a
\emph{flat} subfamily, since these residual members have the same
Hilbert polynomial.

Thus all these members sweep out a \emph{proper} subset in
$G(3,14)$ of dimension $\leqslant 23$. Removing this yields the
needed $U_0$.

\end{proof}

It follows from Lemma~\ref{theorem:s-64-hat-vs-i-xxx} that one can
write $\widehat{\lambda} = \lambda + \lambda^0$ with $\lambda^0$
being parameterized by an algebraic curve $Z$ (over
$\mathbb{C}(\mathcal{F}^{\,\text{R}})$). Similarly, one has
$\widehat{\sigma}_i = \sigma_{i} + \sigma_i^0$, with
$Z$\,-\,parameterized $\sigma_i^0$.

Fix some $\lambda'\in H^0(G(3,7),\bigwedge^3\mathcal{E}_7)$,
$\sigma_i'\in H^0(G(3,7),\bigwedge^2\mathcal{E}_7)$, $1 \le i \le
3$, and a map $f: Z \longrightarrow \mathbb{P}^1$. Then associate
to $S_{64}$ a $\mathrm{K3}$ surface $S'_{64} \subset G(3,7)$ of
genus $12$ given by $\lambda + f(\lambda,z)\lambda' = \sigma_1 +
f(\lambda,z)\sigma'_1 = \ldots = 0$ as usual (for $(\lambda,z)\in
Z$ corresponding to $S_{64}$). This and
Lemma~\ref{theorem:common-open-u} deliver a rational dominant map
$\phi:\mathcal{K}_{33}\dashrightarrow\mathcal{K}_{12}$.

Indeed, the fact that $\phi$ is dominant follows from
Proposition~\ref{theorem:r-polarized} and simple dimension count,
whereas the next result shows that

\begin{lemma}
\label{theorem:cor-def-phi} $\phi$ is correctly defined.
\end{lemma}

\begin{proof}
Note that the collection of $\Lambda$, $\lambda$, $\sigma_i$ is
defined up to the automorphisms in $PGL(14,\mathbb{C})$. Then the
family $\mathcal{F}$ and the locus $U_0$ from
Lemma~\ref{theorem:common-open-u} are defined up to
$PGL(14,\mathbb{C})$ as well. Hence by the Luna's slice theorem it
suffices to show that any $\tau\in PGL(14,\mathbb{C})$, with
$\tau(\mathcal{F}) = \mathcal{F}$ and $\tau(U_0) = U_0$, induces
naturally an automorphism on $G(3,7)$ (thus mapping
$S'_{64}\subset G(3,7)$ to isomorphic surface). But the latter is
evident according to Lemma~\ref{theorem:irred-comps} and the
assumption on $\tau$.

\end{proof}

\begin{proposition}
\label{theorem:shape-of-s-64} $\phi$ is birational.
\end{proposition}

\begin{proof}
Note that the restriction of $\phi$ to the hypersurface
$\mathcal{K}_{33}^{\,\mathrm{R}}$ (cf.
Remark~\ref{remark:direct-proof}) is smooth because by
construction $\phi$ induces an isomorphism there on the tangent
spaces. In addition,
$\phi\big\vert_{\scriptscriptstyle\mathcal{K}_{33}^{\,\mathrm{R}}}$
is one\,-\,to\,-\,one onto its image, and hence it suffices to
show that $\phi^{-1}\phi(S) = S$.

Let $S_{64}\in \phi^{-1}\phi(S)$ be another $\mathrm{K3}$. Then
the corresponding surface $S'_{64} \in \mathcal{K}_{12}$ has
$f(\lambda,z) = 0$ (with notation as above). This gives a rational
dominant map $S_{64}\dashrightarrow S = \Phi_{\scriptscriptstyle
E_3}(S)\subset G(3,7)$ induced by the projection $\pi: G(3,14)
\dashrightarrow G(3,7)$. Here the restriction
$\pi\big\vert_{S_{64}}$ is given by some linear subsystem in
$|\mathcal{O}_{S_{64}}(1)|$. Now, composing
$\pi\big\vert_{S_{64}}$ with embedding $S =
\Phi_{\scriptscriptstyle\widetilde{E}_3}(S) \subset G(3,14)$, we
get a rational projection $S_{64}\dashrightarrow S$ between two
$\mathrm{K3}$ in $\mathbb{P}^{33}$. This is only possible if
$S_{64}$ and $S$ were isomorphic (cf. the proof of
Lemma~\ref{theorem:s-64-hat-vs-i-xxx}).

Thus $\phi^{-1}\phi(S) = S$ as wanted.

\end{proof}

Finally, Theorem~\ref{theorem:main} follows from
Lemma~\ref{theorem:cor-def-phi},
Proposition~\ref{theorem:shape-of-s-64} and
Remark~\ref{remark:s-12-is-unirational}.

\bigskip

\begin{remark}
\label{remark:genus-36} We can run similar arguments as in the
proof of Lemma~\ref{theorem:wedge-e-3} for the vector bundle $E_4$
on $S$ (cf. Remarks~\ref{remark:dual-construction} and
\ref{remark:dual-construction-1}) to show that $E_4 \otimes
\mathcal{O}_{S}(C)$ is a rigid rank $4$ vector bundle such that
$c_1(E_4 \otimes \mathcal{O}_{S}(C)) = H$ and $\chi(S, E_4 \otimes
\mathcal{O}_{S}(C)) = 13$. Then, similarly as in the proof of
Theorem~\ref{theorem:main-1}, one obtains
$$
S \simeq \Phi_{\scriptscriptstyle E_4 \otimes
\mathcal{O}_{S}(C)}(S) \subset G(4, 13) \cap \mathbb{P}^{36}
$$
for $\mathcal{O}_{S}(H) \simeq \mathcal{O}_{G(4,
13)}(1)\big\vert_{S}$ and $S = \Phi_{\scriptscriptstyle E_4
\otimes \mathcal{O}_{S}(C)}(S)$. Now using
Remark~\ref{remark:dual-construction} we prove that $S \subset
G(4, 13) \cap \mathbb{P}^{36}$ coincides with Zariski closure of
the locus $\widetilde{S}\setminus{\Pi}$ inside
$$
\widetilde{S} := \Lambda^* \cap (\lambda^* = 0) \cap (\sigma_{1}^*
= \sigma_{2}^* = \sigma_{3}^* = 0)
$$
for some $\lambda^* \in H^{0}(G(4,
13),\bigwedge^{4}\mathcal{E}_{13})$, $\sigma_{i}^* \in H^{0}(G(4,
13),\bigwedge^{2}\mathcal{Q}_{13}^{*})$, where $\Pi$ and
$\Lambda^*$ are fixed projective subspaces in
$\mathbb{P}(\bigwedge^{4}\mathbb{C}^{13})$, $\mathcal{E}_{13}$ is
the universal vector bundle on $G(4, 13)$, and
$\mathcal{Q}_{13}^*$ is the dual of the universal quotient vector
bundle on $G(4, 13)$. Finally, applying similar arguments as in
the proof of Theorem~\ref{theorem:main}, one can show that the
same holds for any BN general polarized $\mathrm{K3}$ surface of
genus $36$, what again leads to the unirationality of
$\mathcal{K}_{36}$. Let us stress however that the present
strategy does not apply directly to study the unirationality of
$\mathcal{K}_{21}$ and $\mathcal{K}_{28}$ (cf.
Remark~\ref{remark:direct-proof}) because in these two cases there
is no apparent relation between the rigid vector bundles on the
corresponding polarized $\mathrm{K3}$ and those on the genus $12$
surfaces (compare with Lemma~\ref{theorem:wedge-e-3}).
\end{remark}

\end{document}